\documentclass[11pt,reqno]{amsart}
\usepackage{amsrefs}
\usepackage{amsmath}
\usepackage{amsthm}
\usepackage{xfrac}
\usepackage{indentfirst}
\usepackage{mathrsfs}
\usepackage{mathtools}
\usepackage{comment}
\usepackage{amssymb,amsfonts}
\usepackage[
linktocpage=true,colorlinks,citecolor=magenta,linkcolor=blue,urlcolor=magenta]{hyperref}
\usepackage{tikz}
\usetikzlibrary{intersections}
\usepackage{pgfplots}
\usepackage{url}
\newtheorem{thm}{Theorem}[section]
\newtheorem*{thmA}{Theorem A}
\newtheorem*{thmB}{Theorem B}

\newtheorem{cor}[thm]{Corollary}

\newtheorem{lem}[thm]{Lemma}

\newtheorem{prop}[thm]{Proposition}
\theoremstyle{definition}

\newtheorem*{ack}{Acknowledgments}
\newtheorem{prob}{Problem}[section]

\theoremstyle{remark}
\newtheorem{rem}[thm]{Remark}

\numberwithin{equation}{section}
\numberwithin{figure}{section}

\allowdisplaybreaks

\renewcommand{\(}{\left(}
\renewcommand{\)}{\right)}

\newcommand{\mrm}{\mathrm}

\newcommand{\Vol}{\mrm{Vol}}
\newcommand{\Ric}{\mrm{Ric}}

\newcommand{\divv}{\mrm{div}}

\usepackage[margin=1.3in]{geometry}
\newcommand{\metric}[2]{\ensuremath{\langle #1, #2\rangle}}  

\begin{document}
	\title[Heintze-Karcher in warped product manifolds]{New Heintze-Karcher type inequalities in sub-static warped product manifolds}
	
	\author[H. Li]{Haizhong Li}
	\address{Department of Mathematical Sciences, Tsinghua University, Beijing 100084, P.R. China}
	\email{\href{mailto:lihz@tsinghua.edu.cn}{lihz@tsinghua.edu.cn}}
	
	\author[Y. Wei]{Yong Wei}
	\address{School of Mathematical Sciences, University of Science and Technology of China, Hefei 230026, P.R. China}
	\email{\href{mailto:yongwei@ustc.edu.cn}{yongwei@ustc.edu.cn}}
	
	\author[B. Xu]{Botong Xu}
	\address{Department of Mathematics, Technion-Israel Institute of Technology, Haifa 32000, Israel}
	\email{\href{mailto:botongxu@campus.technion.ac.il}{botongxu@campus.technion.ac.il}}


	
\begin{abstract}
	In this paper, we prove Heintze-Karcher type inequalities involving the shifted mean curvature for smooth bounded domains in  certain sub-static warped product manifolds. In particular, we prove a Heintze-Karcher-type inequality for non mean-convex domains in the hyperbolic space.  As applications, we obtain uniqueness results for hypersurfaces satisfying a class of curvature equations.
\end{abstract} 
	\maketitle

\section{Introduction}
Let $\mathbb{R}^{n+1}$ denote the $(n+1)$-dimensional Euclidean space.  By using the moving plane method, Alexandrov \cite{Alex56} proved that geodesic spheres are the only embedded closed hypersurfaces with constant mean curvature in $\mathbb{R}^{n+1}$, which is known as the famous Alexandrov theorem.  Later, by establishing new integral formulas on closed hypersurfaces in $\mathbb{R}^{n+1}$, Reilly \cite{Rei77} gave a new proof of the Alexandrov theorem. Moreover, inspired by the volume estimate by Heintze and Karcher \cite{HK78}, Ros \cite{Ros87} gave another proof of the Alexandrov theorem by establishing the following Heintze-Karcher inequality for closed, mean-convex hypersurfaces in $\mathbb{R}^{n+1}$. Let $\Omega \subset \mathbb{R}^{n+1}$ be a bounded domain with smooth boundary $\Sigma = \partial \Omega$, and assume that the mean curvature $p_1(\kappa)$ of $\Sigma$ is positive. Then
\begin{equation}\label{s1:HK-original}
	\int_{\Sigma} \frac{1}{p_1(\kappa)} d\mu \geq (n+1) \Vol (\Omega).
\end{equation}
Equality holds in \eqref{s1:HK-original} if and only if $\Sigma$ is umbilic, and thus $\Sigma$ is a geodesic sphere.  
	
For the Heintze-Karcher type inequality in space forms, a breakthrough was due to Brendle\cite{Bre13}. Using the unit normal flow with respect to a conformal metric, he proved a Heintze-Karcher type inequality in certain warped product manifolds. In particular, he proved the following inequality in hyperbolic space $\mathbb{H}^{n+1}$ and hemisphere $\mathbb{S}^{n+1}_+$. Let $\Omega$ be a bounded domain with smooth boundary ${\Sigma} = \partial \Omega$ in  $\mathbb{H}^{n+1}$ (resp. $\mathbb{S}^{n+1}_{+}$). Fix a point $p \in \mathbb{H}^{n+1}$ (resp. $\mathbb{S}^{n+1}$), and let $\lambda' = \cosh r(x)$ (resp. $\cos r(x)$), where $r(x) = dist(x, p)$ is the distance function to $p$. Assume that the mean curvature $p_1(\kappa)$ of $\Sigma$ is positive. Then
\begin{equation}\label{s1:HK-Brendle}
	\int_{ \Sigma} \frac{\lambda'}{p_1(\kappa)} d\mu \geq (n+1) \int_{\Omega} \lambda' dv.
\end{equation}
Equality holds in \eqref{s1:HK-Brendle} if and only if $\Sigma$ is umbilic, and thus $\Sigma$ is a geodesic sphere. The inequality \eqref{s1:HK-Brendle} was later proved by using the generalized Reilly's formula, see \cite{LX19, QX15}.  Similar as the Euclidean case, Brendle's Heintze-Karcher inequality was used to prove Alexandrov type theorem for embedded CMC hypersurfaces in warped product spaces \cite{Bre13}. It was also used to prove Minkowski and isoperimetric type inequalities in various warped product manifolds, see for example \cite{BHW16,deLima16,GWWX15,SX19}. 
	
We call an $n$-tuple $\tilde{\kappa} = (\tilde{\kappa}_1, \ldots, \tilde{\kappa}_n)$ the {\em shifted principal curvatures} of a hypersurface $\Sigma$ if there is a constant $\varepsilon$ such that $\tilde{\kappa}_i = \kappa_i -\varepsilon$, $1 \leq i \leq n$, where $\kappa_i$ are the principal curvatures of $\Sigma$. The study of curvature flows and geometric inequalities involving shifted principal curvatures in curved ambient spaces especially in hyperbolic space $\mathbb{H}^{n+1}$ has been attractive in the recent years. See for example \cite{And94,BH17,Lyn22} for fully nonlinear contracting curvature flows with velocity involving shifted principal curvatures $\kappa_i-\varepsilon$ where $\varepsilon$ depends on the sectional curvature lower bound of the ambient manifolds, and  \cite{ACW21, HLW20, LX22, WWZ22} for curvature flows and geometric inequalities in $\mathbb{H}^{n+1}$ involving shifted principal curvatures $\kappa_i-1$. The shifted factor $\varepsilon$ is introduced usually in order to overcome the negative curvature of the background space. 
	
The following Heintze-Karcher type inequality involving the shifted mean curvature was proved by Li-Xu \cite{LX22} in the curve case ($n=1$) and  by Hu-Wei-Zhou \cite{HWZ23} for $n \geq 2$. It was used to obtain a uniqueness result for solutions to the horospherical $p$-Minkowski problem introduced in \cite{LX22}, see \cite[Theorem 1.3]{HWZ23}.
\begin{thmA} [\cite{LX22,HWZ23}] 
	Let $\Omega$ be a bounded domain with smooth boundary $\Sigma = \partial \Omega$ in $\mathbb{H}^{n+1}$. Assume that the mean curvature $p_1(\kappa)$ of $\Sigma$ satisfies $p_1(\kappa) >1$. Then 
	\begin{equation}\label{s1:HK-Hu-Wei-Zhou}
		\int_\Sigma \frac{\lambda' -u}{p_1(\kappa) -1} d\mu \geq (n+1) \int_\Omega \lambda' d v,
	\end{equation}
	where $u = \metric{\sinh r \partial_r}{\nu}$ is the support function of $\Sigma$, and $\nu$ is the outward unit normal of $\Sigma$. Equality holds in \eqref{s1:HK-Hu-Wei-Zhou} if and only if $\Sigma$ is umbilic, and thus $\Sigma$ is a geodesic sphere.
\end{thmA}
The main purpose of this paper is to discover more Heintze-Karcher type inequalities involving the shifted mean curvature in hyperbolic space and in more general sub-static warped product manifolds, and then discuss their applications. 

\subsection{Heintze-Karcher type inequalities for non mean-convex domains in $\mathbb{H}^{n+1}$}

Using the unit normal flow, we first prove the following new Heintze-Karcher type inequality for domains which are not necessarily mean-convex in hyperbolic space.
\begin{thm}\label{thm-+1 HK ineq}
	Let $\Omega$ be a bounded domain with smooth boundary $\Sigma = \partial \Omega$ in $\mathbb{H}^{n+1}$. Assume that the mean curvature $p_1(\kappa)$ of $\Sigma$ satisfies $p_1 (\kappa) >-1$. Then 
	\begin{equation}\label{s1:HK-main-eps=-1}
		\int_\Sigma \frac{\lambda' +u}{p_1(\kappa) +1} d\mu \geq (n+1) \int_\Omega \lambda' d v.
	\end{equation}
	Equality holds in \eqref{s1:HK-main-eps=-1} if and only if $\Sigma$ is umbilic, and thus $\Sigma$ is a geodesic sphere.
\end{thm}
Let $p_k: \mathbb{R}^n \to \mathbb{R}$ denote the normalized $k$th elementary symmetric polynomial (see \eqref{s2:def-pm}). For a given hypersurface, we use $(h_{ij})$ and $(g_{ij})$ to denote its second fundamental form and induced metric, respectively. As an application of Theorem \ref{thm-+1 HK ineq}, we prove the following uniqueness result for the solutions to a class of curvature equations.
\begin{thm}\label{thm-+1 shifted-eq}
	Let $n \geq 1$ and $k \in \{1, \ldots, n\}$. Let $\Omega$ be a bounded domain with smooth boundary $\Sigma=\partial \Omega$ in $\mathbb{H}^{n+1}$. Let $\chi: [1, \infty) \times (-\infty, 0) \to \mathbb{R}$ be a $C^1$-smooth function with $\partial_1 \chi \leq 0$ and $\partial_2 \chi \geq 0$.  Assume that $\Sigma$ satisfies the following equation
	\begin{equation}\label{s1:app-main-eps=-1}
		p_k(\kappa+1) = \chi \( \lambda',  -\lambda' -u \).
	\end{equation}
	\begin{enumerate}
		\item 	If $\partial_2 \chi \not\equiv 0$ and $h_{ij} >- g_{ij}$ on $\Sigma$, then $\Sigma$ is a geodesic sphere. Moreover, if one of the inequalities $\partial_1 \chi \leq 0$ and $\partial_2 \chi \geq 0$ is strict,  then $\Sigma$ is a geodesic sphere centered at $p$.
		\item If $\partial_2 \chi \equiv 0$ and $p_k(\kappa +1)>0$ on $\Sigma$, then $\Sigma$ is a geodesic sphere. Moreover, if $\partial_1 \chi<0$, then $\Sigma$ is a geodesic sphere centered at $p$.
	\end{enumerate}
\end{thm}
Since there always exists some point in a closed hypersurface $\Sigma \subset \mathbb{H}^{n+1}$ at which the principal curvatures are all greater than $1$, the following Alexandrov type theorem is a direct consequence of the second case of Theorem \ref{thm-+1 shifted-eq}.
\begin{cor}\label{s1:thm-1.2}
	Let $n \geq 2$ and $k \in \{2, \ldots, n\}$. Let $\Omega$ be a bounded domain with smooth boundary $\Sigma = \partial \Omega$ in $\mathbb{H}^{n+1}$. If the shifted $k$th mean curvature $p_k(\kappa +1)$ is constant on $\Sigma$, then $\Sigma$ is a geodesic sphere.
\end{cor}
For $k=1$, Corollary \ref{s1:thm-1.2} reduces to the Alexandrov theorem for embedded closed hypersurfaces with constant mean curvature in hyperbolic space \cite{MR91}.

\subsection{Heintze-Karcher type inequalities in sub-static warped product manifolds}

Given \eqref{s1:HK-Brendle}, Theorem A and Theorem \ref{thm-+1 HK ineq}, it is natural to ask if there are Heintze-Karcher type inequalities corresponding to a general shifted factor $\varepsilon$ in certain warped product manifolds so that these results can be regarded as its special cases, see \eqref{s1:shifted HK general-eps-1} below. But it seems a challenge to use the flow method to get such a geometric inequality directly, even in the space forms (see the discussions in Remark \ref{rem-discussion}). However, by imposing  certain  convexity hypotheses on the hypersurfaces and using a different approach, we can prove this kind of Heintze-Karcher type inequalities in certain warped product manifolds, see Theorem \ref{thm-shifted H-K ineq} below.
	
To illustrate our approach, recall the following special case of the Minkowski's second inequality in convex geometry, see e.g. \cite[Theorem 7.2.1]{Sch14}. Let $\Omega$ be a bounded domain with smooth, strictly convex boundary $\Sigma =\partial \Omega$ in $\mathbb{R}^{n+1}$. Then
\begin{equation}\label{s1:min-2-ineq}
	{\rm Area}(\Sigma)^2 \geq (n+1) \Vol(\Omega) \int_\Sigma p_1(\kappa) d\mu.
\end{equation}
Equality holds in \eqref{s1:min-2-ineq} if and only if $\Sigma$ is a round sphere. In this setting, one may observe that inequality \eqref{s1:min-2-ineq} implies \eqref{s1:HK-original}: 
\begin{align}
	\int_{ \Sigma} \frac{1}{p_1(\kappa)} d\mu \geq&\frac{ (n+1) \Vol(\Omega)}{ {\rm Area}(\Sigma)^2\  } \int_\Sigma p_1(\kappa) d\mu \int_{ \Sigma} \frac{1}{p_1(\kappa)} d\mu \nonumber \\
	\geq& (n+1) \Vol(\Omega), \label{s1:original idea}
\end{align}
where we used the Cauchy-Schwarz inequality. 
	
Let $M^{n+1} = [0, \bar{r}) \times N$ be a warped product manifold with metric $\bar{g} =dr^2+\lambda(r)^2 g_N$, where $N$ is an $n$-dimensional closed manifold. By using a generalized Reilly's formula,  Li-Xia \cite{LX19} proved the following Minkowski type inequality (Theorem B) for static-convex domains when the ambient space $M^{n+1}$ is sub-static with potential $\lambda'(r)$ (see Section \ref{subsec-hyper WP}). Here, we call a smooth hypersurface  $\Sigma \subset M^{n+1}$ {\em static-convex} if the inequality (see \cite{BW14})
\begin{equation}\label{s1:static-convex}
	h_{ij} \geq \frac{\bar{\nabla}_\nu \lambda'}{\lambda'} 
\end{equation}
holds everywhere on $\Sigma$, where $\bar{\nabla}$ denotes the connection on $(M^{n+1}, \bar{g})$. 
\begin{thmB}[\cite{LX19}]
	Let $M^{n+1} =[0, \bar{r}) \times N$ $(n \geq 2)$ be a sub-static warped product manifold with metric $\bar{g} =dr^2+\lambda(r)^2 g_N$ and potential $\lambda'(r)$, where $(N, g_N)$ is an $n$-dimensional closed manifold.
	\begin{enumerate}
		\item Let $\Omega \subset M^{n+1}$ be a bounded domain with a connected, static-convex boundary $\Sigma = \partial \Omega$ on which $\lambda'>0$. Then
		\begin{equation}\label{s1:min2-LX19-1}
			\(\int_{ \Sigma} \lambda' d\mu \)^2 \geq (n+1) \int_{\Omega} \lambda' dv \int_{ \Sigma} \lambda' p_1(\kappa) d\mu.
		\end{equation}
		Moreover, if inequality \eqref{s1:static-convex} holds strictly at some point in $\Sigma$, and equality holds in \eqref{s1:min2-LX19-1}, then $\Sigma$ is umbilic and  of constant mean curvature.
		\item Assume $\lambda(r)$ satisfies condition \ref{s2:H condition} (see Section \ref{subsec-hyper WP}). Let $\Omega$ be a bounded domain with $\partial \Omega = \Sigma \cup \partial M$. Assume in addition that $\Sigma$ is static-convex. Then
		\begin{equation}\label{s1:min2-LX19-2}
			\(\int_{ \Sigma} \lambda' d\mu \)^2 \geq (n+1)\(\int_{\Omega} \lambda' dv + \frac{1}{n+1} \lambda(0)^{n+1} \Vol(N, g_N) \) \int_{ \Sigma} \lambda' p_1(\kappa) d\mu.
		\end{equation}
		Moreover, if inequality \eqref{s1:static-convex} holds strictly at some point in $\Sigma$, then equality holds in \eqref{s1:min2-LX19-2} if and only if $\Sigma$ is a slice $\{r\} \times N$ for some $r \in (0, \bar{r})$.
	\end{enumerate}
\end{thmB}
Since space forms are special types of sub-static warped product manifolds,  Theorem B recovers a previous result of Xia \cite{Xia16} when $M^{n+1} \in \{ \mathbb{H}^{n+1}, \mathbb{S}^{n+1}_+\}$. Remark that $\Sigma$ is static-convex when it is horo-convex ($\kappa \geq 1$) and $M^{n+1} = \mathbb{H}^{n+1}$, and when it is convex ($\kappa \geq 0$) and $M^{n+1} =\mathbb{S}^{n+1}_+$ \cite{Xia16}. We also note that the inequality \eqref{s1:static-convex} holds strictly at some point on $\Sigma$ when $M$ is a space form. For other geometric inequalities for static-convex domains, see \cite{BW14, HL22,PY23}.

In the following,  we adapt the idea presented in \eqref{s1:min-2-ineq} - \eqref{s1:original idea} to prove Heintze-Karcher type inequalities involving a general shifted factor $\varepsilon \in \mathbb{R}$.
\begin{thm}\label{thm-shifted H-K ineq}
	Let $M^{n+1} =[0, \bar{r}) \times N$ $(n \geq 2)$ be a sub-static warped product manifold with metric $\bar{g} =dr^2+\lambda(r)^2 g_N$ and potential $\lambda'(r)$, where $N$ is an $n$-dimensional closed manifold.
	\begin{enumerate}
		\item Let $\Omega \subset M^{n+1}$ be a bounded domain with a connected, static-convex boundary $\Sigma = \partial \Omega$ on which $\lambda'>0$.  Let  $\varepsilon$ be a real number such that the following inequality
		\begin{equation}\label{s1:assump in HK}
			\( \lambda'-\varepsilon u \)\( p_1(\kappa) -\varepsilon \) > 0
		\end{equation}
		holds everywhere on $\Sigma$, where $u = \metric{\lambda \partial_r}{\nu}$. Then
		\begin{equation}\label{s1:shifted HK general-eps-1}
			\int_\Sigma \frac{\lambda' - \varepsilon u}{p_1 (\kappa) -\varepsilon}d\mu \geq (n+1) \int_\Omega \lambda' dv.
		\end{equation}
		Moreover, if inequality \eqref{s1:static-convex} holds strictly at some point in $\Sigma$, and equality holds in \eqref{s1:shifted HK general-eps-1}, then $\Sigma$ is umbilic and  of constant mean curvature.
		\item Assume $\lambda(r)$ satisfies condition \ref{s2:H condition}. Let $\Omega$ be a bounded domain with $\partial \Omega = \Sigma \cup \partial M$. Assume in addition that $\Sigma$ is static-convex.  Let  $\varepsilon$ be a real number such that the inequality \eqref{s1:assump in HK} 
		holds everywhere on $\Sigma$. Then
		\begin{equation}\label{s1:shifted HK general-eps-2}
			\int_\Sigma \frac{\lambda' - \varepsilon u}{p_1 (\kappa) -\varepsilon}d\mu \geq (n+1)\int_{\Omega} \lambda' dv +\lambda(0)^{n+1} \Vol(N, g_N). 
		\end{equation}
		Moreover, if inequality \eqref{s1:static-convex} holds strictly at some point in $\Sigma$, then equality holds in \eqref{s1:shifted HK general-eps-2} if and only if $\Sigma$ is a slice $\{r\}\times N$ for some $r \in (0, \bar{r})$.
	\end{enumerate}
\end{thm}
When $M^{n+1}= \mathbb{H}^{n+1}$ and $\varepsilon$ is a constant with $|\varepsilon|\leq 1$, there holds $\lambda'-\varepsilon u \geq \cosh r- |\varepsilon| \sinh r \geq e^{-r}>0$ on $\Sigma$ automatically. In this case, the assumption \eqref{s1:assump in HK} reduces to $p_1(\kappa)>\varepsilon$. 
\begin{rem} For static-convex $\Sigma$,
	\begin{enumerate}
		\item 	 Theorem \ref{thm-shifted H-K ineq} reduces to the results proved by Brendle \cite{Bre13} when $\varepsilon =0$.
		\item  	Theorem \ref{thm-shifted H-K ineq} reduces to Theorem A when $\varepsilon =1$ and $M^{n+1} = \mathbb{H}^{n+1}$.
		\item 	Theorem \ref{thm-shifted H-K ineq} reduces to Theorem \ref{thm-+1 HK ineq} when $\varepsilon =-1$ and $M^{n+1} = \mathbb{H}^{n+1}$.
	\end{enumerate}
\end{rem}
Based on the above remark, we would like to propose the following problem.
\begin{prob}
	When $\varepsilon \notin \{ -1,0,1\}$, is the static-convex assumption of $\Sigma =\partial \Omega$ necessary in Theorem \ref{thm-shifted H-K ineq}, especially when $M^{n+1} =\mathbb{H}^{n+1}$ and $\varepsilon \in (-1,0) \cup (0,1)$?
\end{prob}

We next discuss an application of Theorem \ref{thm-shifted H-K ineq}. Recall that a hypersurface $\Sigma \subset M^{n+1}$ is called {\em strictly shifted $k$-convex} if the shifted principal curvatures $(\kappa -\varepsilon)$ of $\Sigma$ lie in the G{\aa}rding cone $\Gamma^+_k$ (see Section \ref{subsec-ele sym func}). 
\begin{thm}\label{thm-s1:shifted-eq}
	Let $n \geq 2$ and $k \in \{1, \ldots, n\}$. Let $\Omega$ be a bounded domain with smooth, static-convex boundary $\Sigma=\partial \Omega$ in space form $\mathbb{N}^{n+1}(c)$, and assume that $\lambda' -\varepsilon u > 0$ on $\Sigma$, where $\varepsilon \in \mathbb{R}$ is a constant. Let $\chi:\mathbb{R} \times \mathbb{R} \to \mathbb{R}$ be a $C^1$-smooth function with $\partial_1 \chi \leq 0$ and $\partial_2 \chi \geq 0$.  Assume that $\Sigma$ satisfies the following equation
	\begin{equation}\label{s1:shifted-eq}
		p_k(\kappa- \varepsilon) = \chi \(\Phi(r),  \varepsilon \Phi(r) -u \),
	\end{equation}
	 where $\Phi$ is a primitive function of $\lambda$.
	\begin{enumerate}
		\item If $\partial_2 \chi \not\equiv 0$, we further assume that  $h_{ij} >\varepsilon g_{ij}$  on $\Sigma$. Then $\Sigma$ is a geodesic sphere. Moreover, if one of the inequalities $\partial_1 \chi \leq 0$ and $\partial_2 \chi \geq 0$ is strict,  then $\Sigma$ is a geodesic sphere centered at $p$.
		
		\item If $\partial_2 \chi \equiv 0$, we further assume that $\Sigma$ is strictly shifted $k$-convex. Then $\Sigma$  is a geodesic sphere.  Moreover, if $\partial_1 \chi<0$, then $\Sigma$ is a geodesic sphere centered at $p$.
	\end{enumerate}
\end{thm}

\begin{rem} $\ $	
	\begin{enumerate}
		\item The definition domain of $\chi$ can be changed to any domain in $\mathbb{R}^2$ that contains all the possible $(\Phi(r), \varepsilon \Phi(r) -u)$. For example, when $c =-1$, $\varepsilon=1$ and $\Phi(r) =\cosh r$, we have $\Phi(r) \in [1,+\infty)$ and $0<\cosh r- \sinh r \leq \varepsilon \Phi(r) -u<+\infty$. Then we can let $\chi$ be a function defined on $[1, +\infty) \times (0, +\infty )$. 
		\item In the case that $\mathbb{N}^{n+1} = \mathbb{H}^{n+1}$, $\varepsilon =1$ and $\partial_1 \chi =0$, Theorem \ref{thm-s1:shifted-eq} was proved in \cite{HWZ23}. In addition, if $\chi$ is constant and $\Sigma$ is assumed to be horo-convex, then it reduces to \cite[Corollary 1.8]{ACW21}.
	\end{enumerate}
\end{rem}

The paper is organized as follows. In Section \ref{sec-preliminaries}, we collect some basic properties of hypersurfaces in warped product manifolds. In Section \ref{sec-pf thm1.1}, we prove Theorem \ref{thm-+1 HK ineq} by using the unit normal flow in hyperbolic space. In Section \ref{sec-pf thm1.3}, we use Theorem B and the Minkowski formula to prove Theorem \ref{thm-shifted H-K ineq}. In Section \ref{sec-ph thm1.2-1.4}, we use the new Heintze-Karcher type inequalities to prove Theorem \ref{thm-+1 shifted-eq} and Theorem \ref{thm-s1:shifted-eq}.
\begin{ack}
  Part of this work was done during a research visit of B. Xu to the University of Science and Technology of China (USTC). B. Xu would like to thank the School of Mathematical Sciences at USTC, especially Y. Wei, for their hospitality and the excellent research environment provided.
\end{ack}

\section{Preliminaries}\label{sec-preliminaries}
In this section, we collect some basic properties of elementary symmetric polynomials and hypersurfaces in warped product manifolds, especially the properties of hypersurfaces in space forms.
\subsection{Elementary symmetric polynomials} \label{subsec-ele sym func}$ \ $

Given $m \in \{1, \ldots, n\}$, we define the normalized $m$th elementary symmetric polynomial $p_m: \mathbb{R}^n \to \mathbb{R}$ by
\begin{equation}
	p_m (x) = \frac{m!(n-m)!}{n!} \sum_{1 \leq i_1< \cdots< i_m \leq n} x_{i_1} \cdots x_{i_m}, \label{s2:def-pm}
\end{equation}
where $x= (x_1, \ldots, x_n) \in \mathbb{R}^n$. It is easy to see $p_m(1,\ldots,1)=1$. We also set $p_0(x) =1$ and $p_{l} (x) =0$ for $l>n$.  Given a constant $s \in \mathbb{R}$, we denote $x>s$ if $x_i >s$ for all $1 \leq i \leq n$, and we denote $x-s:=(x_1-s,\ldots,x_n-s)\in \mathbb{R}^n$. 

Let $A \in {\rm Sym}(n)$ be an $n \times n$ symmetric matrix. Denote by ${\rm EV}(A)$ the eigenvalues of $A$. By the symmetry of $p_m$, we can define $p_m(A) := p_m (EV(A))$. Then
\begin{equation*}
	p_m (A) = \frac{(n-m)!}{n!} \delta_{i_1, \ldots, i_m}^{j_1, \ldots, j_m} A_{i_1}{}^{j_1} \cdots A_{i_m}{}^ {j_m}, \quad m=1, \ldots, n,
\end{equation*}
where $\delta_{i_1, \ldots, i_m}^{j_1, \ldots, j_m}$ is the generalized Kronecker symbol. Denote 
\begin{equation*}
	(\dot{p}_m)^{i}_j(A) = (\dot{p}_m)^{i}_j ({\rm EV}(A)) = \frac{\partial p_m(A)}{\partial A_i{}^j}.
\end{equation*}
If $A = {\rm diag}\{x_1, \ldots, x_n\}$ is a diagonal matrix, then
\begin{equation*}
	(\dot{p}_m)^{i}_j (A) = \frac{\partial p_m (x)}{\partial x_i} \delta^i_{j}.
\end{equation*}

We have the following basic formulas:
\begin{align}\label{s2:pmij-Aij}
	(\dot{p}_{m})^{i}_j \(A\) A_{i}{}^j =& m p_{m} \(A\), \qquad 
	(\dot{p}_{m})^{j}_j\(A\) \delta_{i}{}^j = mp_{m-1} \(A\).
\end{align}
Define the G{\aa}rding cone by
\begin{equation*}
	\Gamma_m^+ := \{x \in \mathbb{R}^n~|~p_i(x)>0, \ 1 \leq i \leq m  \}.
\end{equation*}
Equivalently, $\Gamma_m^+$ is the connected component of $\{p_m>0\} \subset \mathbb{R}^n$ which contains the positive cone $\Gamma_n^+$.
For $x \in \Gamma_m^+$, we have the following Newton-MacLaurin inequality:
\begin{equation}\label{s2:NM-ineq}
	p_m (x)\leq p_1 (x)p_{m-1}(x),
\end{equation}
with equality if and only if $x_1=\cdots=x_n>0$ when $m\in\{2, \ldots, n\}$. We say matrix $A$ belongs to $\Gamma_m^+$ if ${\rm EV}(A) \in \Gamma^+_m$.  In this case, we have 
\begin{equation}\label{s2:pm-ij-positive}
	(\dot{p}_m)_j^i (A)>0.
\end{equation}

For further properties of elementary symmetric polynomials, we refer readers to \cite{Guan14}.

\subsection{Hypersurfaces in warped product manifolds} \label{subsec-hyper WP}$ \ $

Let $M^{n+1} =  [0, \bar{r}) \times N$ be a warped product manifold with metric
\begin{equation*}
	\bar{g} =dr^2 + \lambda^2(r) g_N,
\end{equation*}
where $N$ is an $n$-dimensional closed manifold, and $\lambda: [0, \bar{r}) \to \mathbb{R}$ is a smooth positive function. Denote by $\bar{\nabla}$ and $\bar{\Delta}$ the Levi-Civita connection and the Laplacian on $M$, respectively. For convenience, we also use the notation $\metric{\cdot}{\cdot}$ for $\bar{g}$. When $M^{n+1}$ has a horizon $\partial M=\{0\}\times N$, we will consider the following condition:
\begin{equation}\tag{\textbf{H}}\label{s2:H condition}
	\lambda'(0) =0, \ \lambda''(0)>0 ~ {\rm and}~\lambda'(r)>0~{\rm for}~r \in (0, \bar{r}).
\end{equation}

For a hypersurface $\Sigma^n \subset M^{n+1}$, we denote by $g$ the induced metric on $\Sigma$, and by $\nabla$ and $\Delta$ the corresponding Levi-Civita connection and Laplacian on $\Sigma$, respectively.

The definition of $p_m$ can be extended to symmetric $(0,2)$ tensor $T_{ij}$ on $\Sigma$ by
\begin{equation*}
	p_m (T, g):= p_m ({\rm EV}(g^{jk} T_{ki})).
\end{equation*}
Denote 
\begin{equation*}
	\dot{p}_m^{ij} ({\rm EV}(g^{-1}T)) = \frac{\partial p_{m} (T,g)}{\partial T_{ij}}.
\end{equation*}

Denote by $\nu$ the outward unit normal of $\Sigma$. Let $\{e_1, \ldots, e_n\}$ be a local orthonormal frame on $\Sigma$, the second fundamental form $h_{ij}$ of $\Sigma$ is defined by $h_{ij} = \metric{\bar{\nabla}_{e_i} \nu}{e_j}$. The principal curvatures $\kappa = \{\kappa_1, \ldots, \kappa_n\}$ are eigenvalues of the Weingarten matrix $(h_i{}^j)$, where $h_i{}^j = h_{ik} g^{kj}$, and $(g^{ij})$ is the inverse matrix of $(g_{ij})$. The mean curvature $p_1(\kappa)$ of $\Sigma$ is given by
\begin{equation*}
	p_1(\kappa) = \frac{1}{n} \sum_{i=1}^n \kappa_i = \frac{1}{n} h_{ij} g^{ij}.
\end{equation*}
For a given constant $\varepsilon \in \mathbb{R}$, we call $\kappa-\varepsilon=(\kappa_1-\varepsilon, \ldots, \kappa_n-\varepsilon)$ the shifted principal curvatures of $\Sigma$, which are eigenvalues of the shifted Weingarten matrix $(h_i{}^j -\varepsilon \delta_i{}^j)$.

Given $\varepsilon \in \mathbb{R}$ and $k \in \{1, \ldots, n\}$, a hypersurface $\Sigma$ is called shifted $k$-convex if $(\kappa -\varepsilon) \in \overline{\Gamma}^+_k$ for all $n$-tuples of shifted principal curvatures along $\Sigma$. It is called strictly shifted $k$-convex if $(\kappa -\varepsilon) \in \Gamma^+_k$ along $\Sigma$, and in this case we have $\dot{p}_m^{ij} (\kappa -\varepsilon) >0$ on $\Sigma$ for all $m=1,\ldots, k$. 

\begin{prop}\label{prop-suff cond of shifted k-convex}
	Let $\Sigma$ be a connected hypersurface in $M^{n+1}$. Assume that $p_k(\kappa -\varepsilon) >0$ on $\Sigma$ and $\kappa-\varepsilon>0$ at some point in $\Sigma$. Then $\Sigma$ is strictly shifted $k$-convex.
\end{prop}
\begin{proof}
	The first assumption shows that the shifted principal curvatures $\kappa-\varepsilon$ of $\Sigma$ lie in a connected component of $\{p_k (x) >0\} \subset \mathbb{R}^{n}$. The second assumption shows that this component contains a point in the positive cone $\Gamma^+_n$. Thus, the connected component is $\Gamma^+_k$, and $\Sigma$ is then strictly shifted $k$-convex.
\end{proof}

The following Lemmas \ref{lem-conf-vf-direvative}--\ref{lem-basic formulas} can be found in \cite{Bre13}.
\begin{lem}\label{lem-conf-vf-direvative}
	The vector field $V := \lambda(r) \partial_r$ satisfies
	\begin{align}
		\metric{\overline{\nabla}_X V}{Y} =& \lambda' (r)  \metric{X}{Y}, \label{s2:conf Killing} \\
		\divv_{M} V =& (n+1) \lambda' (r), \label{s2:div-V}
	\end{align}
	where $X$ and $Y$ are vector fields in $M$, and $\divv_{M} V$ denotes the divergence of $V$ in $(M, \bar{g})$. 
\end{lem}

\begin{lem}\label{lem-basic formulas}
    Let $\Sigma$ be a smooth hypersurface in $M^{n+1}$. Let $\Phi(r)$ be a function satisfying $\Phi'(r) = \lambda(r)$. Denote by $u: = \metric{V}{\nu}$ the support function of $\Sigma$. Then
	\begin{align}
		\nabla_i \Phi =& \metric{V}{e_i}, \label{s2:Phi-i} \\
		\nabla_j \nabla_i \Phi =& \lambda' g_{ij} - u h_{ij}, \label{s2:Phi-ij}\\
		\nabla_i u =& h_i{}^j\metric{V}{e_j}. \label{s2:u-i}
	\end{align}
\end{lem}

The next lemma follows from \eqref{s2:Phi-ij}.
\begin{lem}\label{lem-Min form}
	Let $\Sigma$ be a closed hypersurface in $M$.  For any $\varepsilon \in \mathbb{R}$, we have
\begin{equation}\label{s2:1st-Min-fml}
	\int_{ \Sigma} (\lambda' -\varepsilon u) d\mu = \int_\Sigma u (p_1(\kappa)-\varepsilon) d\mu.
\end{equation}
\end{lem}

Now we explain the terminology ``sub-static warped product manifold" appearing in Theorem B and Theorem \ref{thm-shifted H-K ineq}, see e.g. \cite{LX19}.  We call $M$ a static (resp. sub-static) warped product manifold with potential $\lambda'$ if the Riemannian triple $(M^{n+1}, \bar{g}, \lambda'(r))$ is static (resp. sub-static). Here a Riemannian triple $(\mathcal{M}, \mathfrak{g}, f)$ is called static if
\begin{equation*}
	\Delta_{\mathcal{M}} f \mathfrak{g} -\nabla_{\mathcal{M}}^2 f +f \Ric_\mathcal{M}=0,
\end{equation*}
and it is called sub-static if
\begin{equation*}
	\Delta_{\mathcal{M}} f \mathfrak{g} -\nabla_{\mathcal{M}}^2 f +f \Ric_\mathcal{M} \geq 0,
\end{equation*}
where $f$ is a nontrivial smooth function on $\mathcal{M}$ called the potential function. It was shown in \cite[Proposition 2.1]{Bre13} that $M^{n+1}$ is a sub-static warped product manifold with potential $\lambda'(r)$ if and only if
\begin{align}
	0 \leq& \lambda'(\Ric_N - (n-1) \rho g_N) \nonumber\\
	&+\(\lambda^2 \lambda''' +(n-2)\lambda\lambda'\lambda''+(n-1)\lambda'(\rho -\lambda'^2) \)g_N \label{s2:criterion sub-static}
\end{align}
for some constant $\rho \in \mathbb{R}$. Many important spaces studied in general relativity are sub-static warped product manifolds, see e.g. \cite{Bre13,BW14,BHW16,CLZ19}.

\subsection{Hypersurfaces in space forms} $ \ $

Denote by $\mathbb{N}^{n+1}(-1) = \mathbb{H}^{n+1}$, $\mathbb{N}^{n+1}(0) =\mathbb{R}^{n+1}$ and $\mathbb{N}^{n+1}(1) =\mathbb{S}^{n+1}_+$ the $(n+1)$-dimensional hyperbolic space, Euclidean space and hemisphere, respectively. For $c \in  \{-1,0,1\}$, let $r(x):= dist(x, p)$ denote the distance function to a fixed point $p \in \mathbb{N}^{n+1}(c)$, so that $\mathbb{N}^{n+1}(c)$ can be realized as a warped product manifold with $p= \{r=0\}$ and metric $\bar{g} = dr^2 +\lambda(r)^2 g_{\mathbb{S}^n}$, where
\begin{equation*}
	\lambda(r) = \left\{
	\begin{aligned}
		&\sinh r, \quad &r \in [0, \infty), \quad &{\rm when} \ c=-1, \\
		&r, \quad &r \in [0, \infty), \quad &{\rm when} \ c=0, \\
		&\sin r, \quad &r \in [0, \pi/2), \quad &{\rm when} \ c=1.
	\end{aligned}
	\right.
\end{equation*}
Taking $\rho =1$ in \eqref{s2:criterion sub-static}, one can check directly that $\mathbb{N}^{n+1}(c)$ is a static (sub-static) manifold with potential $\lambda'(r)$. 

Let $\Omega$ be a bounded domain with smooth boundary $\Sigma=\partial \Omega$ in space form $\mathbb{N}^{n+1}$. The shifted second fundamental form $h_{ij} -\varepsilon g_{ij}$ is a Codazzi tensor on $\Sigma$. Consequently, we have (see \cite[Lemma 2.1]{Guan14})
\begin{equation}\label{s2.divfree}
	\nabla_j \dot{p}_m^{ij} (\kappa -\varepsilon)=0, \qquad m= 1,\ldots, n.
\end{equation} 

Now we prove the following shifted version of the higher order Minkowski formula in space forms, see \cite[Lemma 2.3]{HWZ23} for the case $\mathbb{N}^{n+1} =\mathbb{H}^{n+1}$.
\begin{lem}\label{lem-shifted-Minkowski formula}
	Let $\varepsilon$ be a real number. Assume that $\Sigma$ is a smooth, closed hypersurface in $\mathbb{N}^{n+1}(c)$. Then
	\begin{equation} \label{s2:shifted-min-fml}
		\int_\Sigma (\lambda' - \varepsilon u)p_{m-1}(\kappa-\varepsilon) d\mu = \int_\Sigma u p_{m}(\kappa-\varepsilon) d\mu, \quad m=1,\ldots,n.
	\end{equation}
\end{lem}
\begin{proof}
	Using \eqref{s2:Phi-ij} and \eqref{s2:pmij-Aij}, we have
	\begin{align}
		&\dot{p}_{m}^{ij} (\kappa-\varepsilon) \nabla_j \nabla_i \Phi \nonumber\\
		=&	\dot{p}_{m}^{ij} (\kappa-\varepsilon) \( \( \lambda' -\varepsilon  u\)g_{ij} - u\(h_{ij} - \varepsilon g_{ij}\)  \) \nonumber\\
		=& \( \lambda' -\varepsilon  u\) \dot{p}_{m}^{ij} (\kappa-\varepsilon) g_{ij} - u 	\dot{p}_{m}^{ij} (\kappa-\varepsilon)\(h_{ij} - \varepsilon g_{ij}\) \nonumber\\
		=& m\( \lambda' -\varepsilon  u\)p_{m-1} (\kappa-\varepsilon)
		-m up_{m} (\kappa-\varepsilon). \label{p-m+1 s2:Phi-ij}
	\end{align}
	By the divergence-free property \eqref{s2.divfree}, applying integration by parts to \eqref{p-m+1 s2:Phi-ij}, we obtain \eqref{s2:shifted-min-fml}.
\end{proof}

We conclude this section with the evolution equations along a general flow 
\begin{equation}\label{s2:eq-general flow}
	\frac{\partial}{\partial t} X = F \nu
\end{equation}
for hypersurfaces $\Sigma_t$ in space form $\mathbb{N}^{n+1} (c)$.
\begin{lem}\label{lem-evl-eq}
	Along the flow \eqref{s2:eq-general flow} in $\mathbb{N}^{n+1} (c)$, we have
	\begin{align}
		\frac{\partial}{\partial t} (\lambda' -\varepsilon u) =& -c u F -\varepsilon \lambda' F + \varepsilon \metric{\nabla F}{V}, \label{s2:evl-lamb-ep u}\\
        \frac{\partial}{\partial t}h_i^j=&-\nabla^j\nabla_iF-Fh_i^kh_k^j-cF\delta_i^j\label{s2:evl-hij}\\
		\frac{\partial }{\partial t} p_1(\kappa) =& -\frac{1}{n} \Delta F - \frac{|h|^2}{n} F -cF, \label{s2:evl-p1}\\
  \frac{\partial}{\partial t} d\mu_t =& n p_1(\kappa) F d\mu_t. \label{s2:evl-dmu}
	\end{align}
	If $\Sigma_t =\partial \Omega_t$ is closed, then
	\begin{equation}
		\frac{\partial}{\partial t} \int_{\Omega_t} \lambda' dv = \int_{\Sigma_t} \lambda' F d\mu_t. \label{s2:evl-weighted vol}
	\end{equation}
\end{lem}
\begin{proof}
	Note that $\lambda'' + c \lambda =0$ in $\mathbb{N}^{n+1}(c)$. Then
	\begin{equation*}
		\frac{\partial}{\partial t} \lambda' = \langle \bar{\nabla}\lambda',\partial_tX\rangle=-c \metric{V}{\partial_t X} = -cuF. 
	\end{equation*}
	The unit normal evolves by $\partial_t\nu=-\nabla F$ (see \cite[(3.17))]{And94}.  Using \eqref{s2:conf Killing}, we have
	\begin{align*}
		\frac{\partial}{\partial t} u = \langle \bar{\nabla}_{\partial_tX}V, {\nu}\rangle +\langle V,\partial_t\nu\rangle =\lambda' F - \metric{V}{\nabla F}.
	\end{align*}
	Combining the above two equations, we obtain \eqref{s2:evl-lamb-ep u}. The equation \eqref{s2:evl-hij} can be found in \cite[(3.19)]{And94}, and \eqref{s2:evl-p1} follows from taking the trace of \eqref{s2:evl-hij}. The equation \eqref{s2:evl-dmu} follows from the standard variation formula of area form.  The last equation \eqref{s2:evl-weighted vol} follows from the co-area formula. 
\end{proof}

\section{Proof of Theorem \ref{thm-+1 HK ineq}} \label{sec-pf thm1.1}
In this section, we prove the Heintze-Karcher type inequality involving the shifted mean curvature $p_1(\kappa +1)$ stated in Theorem \ref{thm-+1 HK ineq}.
\begin{lem}\label{lem-evl eq}
	Consider flow \eqref{s2:eq-general flow} in space form $\mathbb{N}^{n+1} (c)$. Let $\varepsilon \in \mathbb{R}$ be a constant such that $(\lambda' -\varepsilon u) F <0$ and $p_1(\kappa) -\varepsilon \neq 0$ on $\Sigma_t$. Then we have
	\begin{align}
		\partial_t \( \frac{\lambda' -\varepsilon u}{p_1 (\kappa) -\varepsilon} d\mu_t\)
		\leq& (n+1) \varepsilon  \(\frac{\lambda' -\varepsilon u}{p_1 (\kappa) -\varepsilon} - u\) Fd\mu_t +(n+1) \lambda' Fd\mu_t \nonumber \\
		&+\mathcal{T} d\mu_t, \label{s3:evl-Q-main}
	\end{align}		
	where
	\begin{equation*}
		\mathcal{T}:=  \frac{(\lambda'-\varepsilon u) \Delta F}{n(p_1(\kappa) -\varepsilon)^2} +\frac{\varepsilon\metric{\nabla F}{V} }{p_1(\kappa) -\varepsilon} + \frac{\varepsilon^2+c}{p_1(\kappa) -\varepsilon}   \( \frac{\lambda'-\varepsilon u}{p_1(\kappa) -\varepsilon} -u \)F.
	\end{equation*}
	Equality holds in \eqref{s3:evl-Q-main} if and only if $\Sigma_t$ is umbilic at the point we calculate. Moreover, if $\Sigma_t = \partial \Omega_t$ is the boundary of a smooth bounded domain $\Omega_t$, then 
	\begin{align}
		&\frac{\partial}{\partial t} \( \int_{\Sigma_t} \frac{\lambda' -\varepsilon u}{p_1 (\kappa) -\varepsilon} d\mu_t- (n+1) \int_{\Omega_t} \lambda' dv\) \nonumber \\
		\leq& (n+1) \varepsilon \int_{\Sigma_t}  \(\frac{\lambda' -\varepsilon u}{p_1 (\kappa) -\varepsilon} - u\) Fd\mu_t +\int_{\Sigma_t}\mathcal{T} d\mu_t.  \label{s3:evl-Q-extra}
	\end{align}
\end{lem}

\begin{proof}
	Using \eqref{s2:evl-lamb-ep u}, \eqref{s2:evl-p1}, the assumption $\(\lambda'-\varepsilon u\) F <0$ and the trace inequality $|h|^2 \geq np_1(\kappa)^2$, we have
	\begin{align}
		&\frac{\partial}{\partial t}  \( \frac{\lambda'-\varepsilon u}{p_1 (\kappa) -\varepsilon} \)  \nonumber\\
		=& \frac{\partial_t \( \lambda' -\varepsilon u \)}{p_1(\kappa) -\varepsilon} - \frac{\lambda'-\varepsilon u}{\( p_1(\kappa)-\varepsilon  \)^2}\frac{\partial}{\partial t} p_1(\kappa) \nonumber\\
		=& \frac{-cuF - \varepsilon \lambda' F +\varepsilon \metric{\nabla F}{V}}{p_1 (\kappa) -\varepsilon}
		+ \frac{\lambda'-\varepsilon u}{(p_1 (\kappa) -\varepsilon)^2} \(\frac{\Delta F}{n} +\frac{|h|^2}{n} F +cF\) \nonumber\\
		\leq & \frac{- \varepsilon (\lambda'-\varepsilon u)F -(\varepsilon^2+c) uF +\varepsilon \metric{\nabla F}{V} }{p_1(\kappa) -\varepsilon} \nonumber\\
		&+  \frac{\lambda'-\varepsilon u}{(p_1 (\kappa) -\varepsilon)^2} \( \frac{\Delta F}{n} + (p_1 (\kappa)-\varepsilon)^2 F + 2 \varepsilon(p_1(\kappa) -\varepsilon) F + (\varepsilon^2+c)F  \) \nonumber\\
		=&\varepsilon \frac{\lambda'-\varepsilon u}{p_1(\kappa) -\varepsilon} F+(\lambda'-\varepsilon u)F +  \frac{\varepsilon\metric{\nabla F}{V} }{p_1(\kappa) -\varepsilon} +\frac{(\lambda'-\varepsilon u) \Delta F}{n(p_1(\kappa) -\varepsilon)^2} \nonumber \\
		&+ \frac{\varepsilon^2+c}{p_1(\kappa) -\varepsilon} \( \frac{\lambda'-\varepsilon u}{p_1(\kappa) -\varepsilon} -u \)F. \label{dt lam-u/p1-eps}
	\end{align}
	Note that equality holds in \eqref{dt lam-u/p1-eps} if and only if $\Sigma_t$ is umbilic at the point. This together with \eqref{s2:evl-dmu} gives
	\begin{align*}
		&\frac{\partial}{\partial t}  \( \frac{\lambda' -\varepsilon u}{p_1 (\kappa) -\varepsilon} d\mu_t \) \\
		\leq& \left( \varepsilon \frac{\lambda'-\varepsilon u}{p_1(\kappa) -\varepsilon} F+(\lambda'-\varepsilon u)F +  \frac{\varepsilon\metric{\nabla F}{V} }{p_1(\kappa) -\varepsilon} +\frac{(\lambda'-\varepsilon u) \Delta F}{n(p_1(\kappa) -\varepsilon)^2} \right. \\
		&\left.+ \frac{\varepsilon^2+c}{p_1(\kappa) -\varepsilon} \( \frac{\lambda'-\varepsilon u}{p_1(\kappa) -\varepsilon} -u \)F +  \frac{\lambda'-\varepsilon u}{p_1(\kappa) -\varepsilon} \(n(p_1(\kappa) - \varepsilon) +n \varepsilon \) F \right)d\mu_t\\\
		=& \left( (n+1)\varepsilon \frac{\lambda'-\varepsilon u}{p_1(\kappa) -\varepsilon} F+ (n+1)(\lambda' -\varepsilon u) F
		+ \frac{\varepsilon\metric{\nabla F}{V} }{p_1(\kappa) -\varepsilon}  \right. \\
		&\left. +\frac{(\lambda'-\varepsilon u) \Delta F}{n(p_1(\kappa) -\varepsilon)^2}+ \frac{\varepsilon^2+c}{p_1(\kappa) -\varepsilon} \( \frac{\lambda'-\varepsilon u}{p_1(\kappa) -\varepsilon} -u \)F \right) d\mu_t\\
		=&\biggl( (n+1) \varepsilon \( \frac{\lambda'-\varepsilon u}{p_1(\kappa) -\varepsilon} -u \) F  +(n+1) \lambda' F + \frac{\varepsilon\metric{\nabla F}{V} }{p_1(\kappa) -\varepsilon} \nonumber\\ 
        & +\frac{(\lambda'-\varepsilon u) \Delta F}{n(p_1(\kappa) -\varepsilon)^2} +\frac{\varepsilon^2 +c}{p_1(\kappa) -\varepsilon} \(  \frac{\lambda'-\varepsilon u}{p_1(\kappa) -\varepsilon} -u \) F\biggr) d\mu_t.
	\end{align*}
	Thus we obtain \eqref{s3:evl-Q-main}. Then \eqref{s3:evl-Q-extra} follows from \eqref{s3:evl-Q-main} and \eqref{s2:evl-weighted vol}. We complete the proof of Lemma \ref{lem-evl eq}.
\end{proof}

\begin{rem}\label{rem-discussion}
	Here, we briefly discuss the difficulty of proving the Heintze-Karcher type inequality involving a general shifted factor $\varepsilon$ by using the flow method. One can show directly that the hard term $\mathcal{T}$ vanishes in the following cases:
	\begin{enumerate}
		\item $c=-1$, $\varepsilon=1$ and $F=-1$,
		\item $c = \pm 1$, $\varepsilon =0$ and $F = -\lambda'$.
	\end{enumerate}
	When $c =-1$, $\varepsilon=1$ and  $F= -1$, the unit normal flow in hyperbolic space was used by Hu-Wei-Zhou \cite{HWZ23} to prove Theorem A. By using \eqref{s2:evl-hij}, it is not difficult to see that the flow preserves the class of horospheres ($\kappa \equiv 1$). That is, $\Sigma_t$ ($t>0$)  is a horosphere if the initial data $\Sigma_0$ is a horosphere. When $c = \pm 1$, $\varepsilon=0$ and $F =-\lambda'$, this flow was used by Brendle \cite{Bre13} to prove \eqref{s1:HK-Brendle}. By using \eqref{s2:evl-hij}, we see that it preserves the class of geodesic hyperplanes ($\kappa \equiv 0$). By a direct calculation, we find that the following flow preserves the class of equidistant hypersurfaces with constant principal curvature $\varepsilon$ in $\mathbb{H}^{n+1}$: 
	\begin{equation*}
		\partial_t X=F \nu, \quad 	F = - \( \lambda'-\varepsilon u \) \xi (u -\varepsilon \lambda'),
	\end{equation*}
	where $\xi$ is an arbitrary $C^2$ function. In this case, the coefficient of $\nabla p_1$ in further expansions of $\mathcal{T}$ vanishes if and only if $\xi$ satisfies the equation
	\begin{equation*}
		\frac{d}{dt}\log \xi(t) =  \frac{\varepsilon}{(1-\varepsilon^2) \lambda' -\varepsilon t}.
	\end{equation*}
	The above equation can be solved when $\varepsilon (1-\varepsilon^2) =0$.  After a direct calculation, we find that considering the flows in space forms studied in \cite{Bre13} and \cite{HWZ23}, the remaining manageable case is that
 \begin{enumerate}
		\item[(3)]  $c=-1$, $\varepsilon=-1$ and $F=-1$.
	\end{enumerate}
 Note also that, when $F$ is a constant, the term $\mathcal{T}$ vanishes if and only if $\varepsilon^2+c =0$. Hence, it is difficult to prove Heintze-Karcher type inequality involving a non-zero shifted factor $\varepsilon$ by using the unit normal flow in Euclidean space and hemisphere.
\end{rem}

Let $\Omega$ be a bounded domain with smooth boundary $\Sigma = \partial \Omega$ in $\mathbb{H}^{n+1}$, and assume that the mean curvature of $\Sigma$ satisfies $p_1(\kappa)>-1$. Inspired by the discussion in the above Remark \ref{rem-discussion}, let us consider the normal exponential map  $X: \Sigma \times [0, \infty) \to \bar{\Omega}$ given by $X(x,t)=\exp_x(-t\nu(x))$. For each point $q \in \bar{\Omega}$, let $w(q) = \text{dist}(q, \Sigma)$. Define
\begin{align*}
	A =& \{ (x,t) \in \Sigma \in [0, \infty) ~:~w\( X(x,t) \) =t \},\\
	A^* =&\{ (x,t) \in \Sigma \times [0, \infty): (x, t+\delta) \in A \ {\rm for \ some} \ \delta>0 \}.
\end{align*}
It is known that the sets $A$ and $A^*$ have the following properties \cite{Bre13}:
\begin{enumerate}
	\item If $(x, t_0) \in A$, then $(x,t) \in A$ for all $t \in [0, t_0]$.
	\item The set $A$ is closed, and we have $X(A) = \bar{\Omega}$.
	\item The set $A^*$ is an open subset of $\Sigma \times [0, +\infty)$, and the restriction $X|_{A^*}$ is a diffeomorphism.
\end{enumerate}
Note that the cut locus $C$ of $\Sigma$ has finite $n$-dimensional Hausdorff measure\cite{LN05}. Since $\Sigma$ is compact, and the distance from the points in $\Sigma$ to their cut points in $C$ is a Lipschitz continuous function on $\Sigma$ \cite{LN05}, there exists $0 \leq T <\infty$ such that $0 \leq w(q) \leq T$ for all $q \in \bar{\Omega}$. For each $t \in [0, T]$, define
\begin{equation*}
	\Sigma_t^* = X \(A^* \cap (\Sigma \times \{t \}) \).
\end{equation*}
It holds that $\Sigma_t^* = \{w=t \} \setminus C$  is a smooth hypersurface which is contained in the level set $\{w =t\}$, and $\Omega \setminus C = X(A^*) = \cup_{0<t<T} \Sigma_t^*$. It is easy to see that $\Sigma_t^*$ satisfies the unit normal flow 
\begin{equation}\label{s3.ptX}
    \partial_t X (y,t)= -\nu (y,t),
\end{equation}
where $\nu = -\bar{\nabla} w / |\bar{\nabla} w|$ denotes the outward unit normal of $\Sigma_t^*$.

\begin{lem}\label{lem-finite area}
	The mean curvature of $\Sigma_t^*$ satisfies $p_1(\kappa)>-1$. The area  of $\Sigma_t^*$ satisfies $\mu (\Sigma_t^*) < e^{nt} \mu (\Sigma)< e^{nT}\mu(\Sigma)$ for all $t \in [0, T)$.
\end{lem}
\begin{proof}
	Taking $F=-1$ and $c=-1$ in \eqref{s2:evl-p1} and using $|h|^2 \geq n p_1(\kappa)^2$, at each point in $\Sigma_t^*$ we have
	\begin{equation*}
		\frac{\partial }{\partial t} (p_1(\kappa)+1)= \frac{|h|^2}{n}-1 \geq  (p_1(\kappa) +1) \( p_1(\kappa) -1 \).
	\end{equation*}
	Using the parabolic maximum principle, we have $p_1(\kappa) > -1$ on $\Sigma_t^*$. Then \eqref{s2:evl-dmu} yields that the area form $d\mu_t$ of $\Sigma_t^*$ satisfies $\partial_t d\mu_t < n d\mu_t.$ Then the second assertion follows.
\end{proof}

The following crucial lemma was proved by Hu-Wei-Zhou \cite{HWZ23}.
\begin{lem}[\cite{HWZ23}]\label{lem-HWZ}
	For a.e. $t \in (0,T)$, we have
	\begin{equation*}
		\int_{\Sigma_t^*} u d\mu_t  = (n+1) \int_{ \{w>t \} } \lambda' dv. 
	\end{equation*}
\end{lem}

Now we introduce the quantity
\begin{equation*}
	Q(t) = e^{-(n+1)  t}\( \int_{\Sigma_t^*} \frac{\lambda'+u}{p_1(\kappa) +1} d\mu_t -(n+1) \int_{ \{w>t \} } \lambda' dv \).
\end{equation*}

\begin{lem}\label{lem-mono Q}
	For any $t \in [0,T)$, there holds $Q(0) \geq Q(t)$.
\end{lem}
\begin{proof}
	Since $\Sigma_t^*$ satisfies the unit normal flow \eqref{s3.ptX}, the co-area formula implies that
	\begin{equation}
		\int_{ \{w>t \} } \lambda' dv = \int_t^T\int_{\Sigma_{\tau}^*} \lambda' d\mu_\tau d\tau. \label{s3:co-area}
	\end{equation}
	Then 
	\begin{equation*}
		\frac{d}{dt} \int_{ \{w>t\} } \lambda' dv = - \int_{\Sigma_t^*} \lambda' d\mu_t.
	\end{equation*}
	Note that $\lambda' +u  = \lambda' + \metric{\lambda \partial_r}{\nu}\geq \lambda'-\lambda >0$ and $p_1(\kappa) +1>0$ on $\Sigma_t^*$. Taking $c=-1$, $\varepsilon =-1$ and $F =-1$ in the evolution equation \eqref{s3:evl-Q-main}, we have
	\begin{align*}
		&\limsup_{h \searrow 0} \frac{1}{h} (Q(t) -Q(t-h))\\
		\leq& -(n+1) Q(t)  +e^{-(n+1) t} \(   \int_{\Sigma_t^*} \partial_t \(\frac{\lambda'+u}{p_1(\kappa) +1} d\mu_t  \) +(n+1) \int_{\Sigma_t^*} \lambda' d\mu_t  \)\\
		\leq& -(n+1) Q(t) + (n+1) e^{-(n+1) t} \int_{\Sigma_t^*} \(  \frac{\lambda'+ u}{p_1(\kappa) +1} -u  \) d\mu_t\\
		=& (n+1) e^{-(n+1) t}  \( (n+1)\int_{ \{w>t \} } \lambda' dv- \int_{\Sigma_t^*} u d\mu_t\).
	\end{align*}
	It follows from Lemma \ref{lem-HWZ} that the last line vanishes for a.e. $t \in (0,T)$. Then for any $t \in (0,T)$, we have
	\begin{align*}
		Q(t) - Q(0) \leq& \int_0^t (n+1) e^{-(n+1) \tau}  \( (n+1)\int_{ \{w>\tau \} } \lambda' dv- \int_{\Sigma_{\tau}^*} u d\mu_\tau \) d\tau\\
		=& 0.
	\end{align*}
	We complete the proof of Lemma \ref{lem-mono Q}.
\end{proof}

\begin{proof}[Proof of Theorem \ref{thm-+1 HK ineq} ]
	Note that $\Sigma = \Sigma_0^*$ and $\Omega = \{w>0 \}$. It follows from Lemma \ref{lem-mono Q} that
	\begin{equation}\label{s3:Q_0 geq Q_inf}
		\int_\Sigma \frac{\lambda' +u}{p_1(\kappa) +1} d\mu - (n+1) \int_\Omega \lambda' d v = Q(0) \geq \liminf_{t \to T} Q(t).
	\end{equation}
	On the other hand, since $\lambda'+u>0$ and $p_1(\kappa) +1>0$ on $\Sigma_t^*$ for all $t \in [0,T)$, we have from \eqref{s3:co-area} that
	\begin{equation*}
		Q(t) \geq -(n+1) e^{-(n+1) t} \int_{ \{w>t \} } \lambda' dv \geq  -(n+1) \int_t^T \int_{\Sigma_{\tau}^*} \lambda' d\mu_\tau d\tau.
	\end{equation*}
	Then the uniform upper bound of $\mu (\Sigma_t^*)$ given in Lemma \ref{lem-finite area} implies 
	\begin{equation}\label{s3:liminf-Q}
		\liminf_{t \to T} Q(t) \geq 0.
	\end{equation}
	Combining \eqref{s3:Q_0 geq Q_inf} and \eqref{s3:liminf-Q}, we obtain the desired Heintze-Karcher type inequality \eqref{s1:HK-main-eps=-1}. If equality holds in \eqref{s1:HK-main-eps=-1}, then Lemma \ref{lem-evl eq} implies that $\Sigma$ is a closed, umbilic hypersurface in $\mathbb{H}^{n+1}$, and thus $\Sigma$ is a geodesic sphere. We complete the proof of Theorem \ref{thm-+1 HK ineq}.
\end{proof}

\section{Proof of Theorem \ref{thm-shifted H-K ineq}}\label{sec-pf thm1.3}
In this section, we prove the Heintze-Karcher type inequality involving the shifted mean curvature with a general shifted factor $\varepsilon$ stated in Theorem \ref{thm-shifted H-K ineq}.

If $\partial \Omega =\Sigma$, then we can use the divergence theorem and \eqref{s2:div-V} to get
	\begin{equation}\label{s4:div fml-no horizon}
		\int_\Sigma u d\mu = \int_\Sigma \metric{V}{\nu} d\mu 
		= \int_\Omega \divv_{M} V dv =(n+1) \int_\Omega \lambda' dv.  
	\end{equation}
	If $M$ has a horizon $\partial M =\{0\} \times N$, then $\bar{g} |_{\partial M} = \lambda(0)^2 g_N$ and $\metric{\partial_r}{\nu} = -1$ on $\partial M$. If $\partial\Omega=\Sigma\cup\partial M$, we have
	\begin{align}
		\int_\Sigma u d\mu =&(n+1) \int_\Omega \lambda' dv - \int_{\{0\} \times N} \metric{\lambda(0) \partial_r}{\nu} d\mu \nonumber \\
		=&(n+1) \int_\Omega \lambda' dv  +  \lambda(0)^{n+1} \Vol(N, g_N). \label{s4:div fml-has horizon}
	\end{align}
	Thus, both \eqref{s1:min2-LX19-1} and \eqref{s1:min2-LX19-2} in Theorem B can be written as
	\begin{equation}\label{s4:general min-2 ineq}
		\( \int_{ \Sigma} \lambda' d\mu \)^2 \geq	\int_\Sigma u d\mu\int_\Sigma \lambda' p_1(\kappa) d\mu.
	\end{equation}
	Using the Minkowski formula \eqref{s2:1st-Min-fml}, we have
	\begin{align}
		&	\int_\Sigma u d\mu\int_\Sigma \lambda' p_1(\kappa) d\mu  \nonumber \\
		=& 	\int_\Sigma u d\mu \int_\Sigma \((\lambda' - \varepsilon u) + \varepsilon u\) \(p_1(\kappa-\varepsilon) + \varepsilon  \) d\mu  \nonumber\\
		=&	\int_\Sigma u d\mu \int_\Sigma (\lambda' - \varepsilon u)p_1(\kappa-\varepsilon) d\mu
		+ \varepsilon^2 \(  \int_\Sigma u d\mu \)^2  \nonumber \\
		&+ \int_\Sigma u d\mu \int_\Sigma \varepsilon \( (\lambda' - \varepsilon u) + u p_1(\kappa-\varepsilon) \) d\mu \nonumber \\
		=&	\int_\Sigma u d\mu \int_\Sigma (\lambda' - \varepsilon u)p_1(\kappa-\varepsilon) d\mu + \varepsilon^2 \(  \int_\Sigma u d\mu \)^2 \nonumber\\
		&+2 \varepsilon 	\int_\Sigma u d\mu 	 \int_\Sigma \(\lambda'-\varepsilon u\) d\mu. \label{s4:rhs of LX19}
	\end{align}
	Meanwhile, we compute
	\begin{align}
		\( \int_\Sigma \lambda' d\mu \)^2
		=& \( \int_\Sigma \(\(\lambda'-\varepsilon u\) + \varepsilon u\) d\mu \)^2  \nonumber\\
		=& 	\( \int_\Sigma \(\lambda'-\varepsilon u\) d\mu \)^2 + \varepsilon^2	\( \int_\Sigma u d\mu \)^2 \nonumber\\
		&+ 2\varepsilon \int_\Sigma \(\lambda' -\varepsilon u \)d\mu \int_\Sigma u d\mu. \label{s4:lhs of LX19}
	\end{align}
	Inserting \eqref{s4:rhs of LX19} and \eqref{s4:lhs of LX19} into \eqref{s4:general min-2 ineq}, we have
	\begin{equation}\label{equiv-Xia16-ineq}
		\( \int_\Sigma \(\lambda'-\varepsilon u\) d\mu \)^2
		\geq 	\int_\Sigma u d\mu \int_\Sigma (\lambda' - \varepsilon u)p_1(\kappa-\varepsilon) d\mu.
	\end{equation}
	On the other hand, by using the assumption \eqref{s1:assump in HK} and the Cauchy-Schwarz inequality, we get
	\begin{equation}\label{s4:cau-sch ineq}
		\( \int_\Sigma \(\lambda'-\varepsilon u\) d\mu \)^2
		\leq \int_\Sigma \(\lambda'-\varepsilon u\) p_1(\kappa-\varepsilon) d\mu \int_\Sigma  \frac{\lambda'-\varepsilon u}{p_1(\kappa-\varepsilon)} d\mu.
	\end{equation}
	Note that equality holds in \eqref{s4:cau-sch ineq} if and only if $p_1(\kappa)$ is constant on $\Sigma$. Combining \eqref{equiv-Xia16-ineq} with \eqref{s4:cau-sch ineq}, we have
	\begin{equation*}
		\int_\Sigma  \frac{\lambda'-\varepsilon u}{p_1(\kappa-\varepsilon)} d\mu
		\geq \int_\Sigma u d\mu.
	\end{equation*}
	Then inequalities \eqref{s1:shifted HK general-eps-1} and \eqref{s1:shifted HK general-eps-2} follow by inserting \eqref{s4:div fml-no horizon} and \eqref{s4:div fml-has horizon} into the above inequality, respectively. The cases that equalities hold in these inequalities inherit from Theorem B. We complete the proof of Theorem \ref{thm-shifted H-K ineq}.

\section{Proofs of Theorem \ref{thm-+1 shifted-eq} and Theorem \ref{thm-s1:shifted-eq}} \label{sec-ph thm1.2-1.4}
In this section, we prove the uniqueness result of solutions to a class of curvature equations stated in Theorem \ref{thm-+1 shifted-eq} and Theorem \ref{thm-s1:shifted-eq} by using the Heintze-Karcher type inequalities established in Theorem \ref{thm-+1 HK ineq} and Theorem \ref{thm-shifted H-K ineq}, respectively.
Since their proofs are similar, we first prove Theorem \ref{thm-s1:shifted-eq} and then point out the differences between this proof and that of Theorem \ref{thm-+1 shifted-eq}.

\begin{proof}[Proof of Theorem \ref{thm-s1:shifted-eq}]
	For abbreviation, we denote
	\begin{equation*}
		\dot{p}_k^{ij} : =	\dot{p}_k^{ij}(\kappa-\varepsilon) = \frac{\partial p_k (\kappa -\varepsilon)}{\partial (h_{ij} -\varepsilon g_{ij})}.
	\end{equation*}
    Since $\Sigma$ is strictly shifted $k$-convex, we have $\chi >0$ and that $\dot{p}_k^{ij}$ is positive definite on $\Sigma$.
    Using \eqref{p-m+1 s2:Phi-ij}, \eqref{s1:shifted-eq} and integration by parts, we have
	\begin{align}
		&\int_\Sigma \(\(\lambda' -\varepsilon u\) \frac{p_{k-1}  (\kappa-\varepsilon)}{p_k (\kappa-\varepsilon)} -u \)d\mu \nonumber\\
		=& \int_\Sigma \frac{(\lambda' - \varepsilon u) p_{k-1} (\kappa-\varepsilon) - up_k (\kappa-\varepsilon)}{p_k (\kappa-\varepsilon)} d\mu\nonumber\\
		=& \frac{1}{k}\int_\Sigma \frac{\dot{p}_k^{ij}\nabla_j \nabla_i \Phi(r)  }{\chi} d\mu \nonumber\\
		=&- \frac{1}{k}\int_\Sigma \dot{p}_k^{ij} \nabla_i \Phi(r) \nabla_j \( \frac{1}{\chi}\) d\mu \nonumber\\
		=&\frac{1}{k}\int_\Sigma \frac{ \dot{p}_k^{ij}}{\chi^2} \nabla_i \Phi \( \partial_1 \chi \nabla_j \Phi+\partial_2 \chi \nabla_j \( \varepsilon \Phi -u \) \) d\mu,\label{s5:integ by part}
	\end{align}
	 where in the third equality we used the divergence free property \eqref{s2.divfree} of $\dot{p}_k^{ij}$. For any fixed point $q \in \Sigma$, we take an orthogonal basis $\{e_1, \ldots, e_n\}$ at $q$ such that $g_{ij} = \delta_{ij}$ and $h_{ij} = \kappa_i \delta_{ij}$. Recall the assumptions $\partial_1 \chi \leq 0$ and $\partial_2 \chi \geq 0$. Using \eqref{s2:Phi-i} and \eqref{s2:u-i}, we have
	\begin{equation}\label{rigid-chi_1}
		\partial_1 \chi\dot{p}_k^{ij} \nabla_i \Phi \nabla_j \Phi =  \partial_1 \chi\sum_{i=1}^n  \dot{p}_k^{ii}  \metric{V}{e_i}^2 \leq 0
	\end{equation}
 and 
 	\begin{align}
		\partial_2 \chi\dot{p}_k^{ij} \nabla_i \Phi \nabla_j \(\varepsilon \Phi -u\)
		=& \partial_2 \chi\dot{p}_k^{ij} \metric{V}{e_i} \(\varepsilon \metric{V}{e_j}- h_j{}^l\metric{V}{e_l}\) \nonumber\\
		=&- \partial_2 \chi\sum_{i=1}^n \dot{p}_k^{ii}  \metric{V}{e_i}^2 (\kappa_i -\varepsilon).  \label{rigid-chi_2}
	\end{align}
 Inserting \eqref{rigid-chi_1} and \eqref{rigid-chi_2} into \eqref{s5:integ by part}, we obtain
	\begin{align}\label{s5:integral leq 0}
		& \int_\Sigma \(\(\lambda' -\varepsilon u\) \frac{p_{k-1}  (\kappa-\varepsilon)}{p_k (\kappa-\varepsilon)} -u \)d\mu \nonumber \\
  \leq  &-\frac{1}{k}\int_{\Sigma}\frac{\partial_2\chi}{\chi^2}\sum_{i=1}^n \dot{p}_k^{ii}  \metric{V}{e_i}^2 (\kappa_i -\varepsilon).
	\end{align}
When $\partial_2 \chi \equiv 0$, the right hand side of  \eqref{s5:integral leq 0} vanishes; when  $\partial_2 \chi \not\equiv 0$, the assumption $h_{ij} > \varepsilon g_{ij}$ implies that the right hand side of \eqref{s5:integral leq 0} is non-positive. That is, we have 
\begin{align}\label{s5:integral leq 0b}
	 \int_\Sigma \(\(\lambda' -\varepsilon u\) \frac{p_{k-1}  (\kappa-\varepsilon)}{p_k (\kappa-\varepsilon)} -u \)d\mu \leq 0
	\end{align}
in both cases. 

	On the other hand, by using the Newton-MacLaurin inequality \eqref{s2:NM-ineq}, the assumption $\lambda'-\varepsilon u>0$, \eqref{s4:div fml-no horizon} and Theorem \ref{thm-shifted H-K ineq}, we have
	\begin{align}
		&\int_\Sigma \(\(\lambda' -\varepsilon u\) \frac{p_{k-1}  (\kappa-\varepsilon)}{p_k (\kappa-\varepsilon)} -u\) d\mu \nonumber\\ \geq&
		\int_\Sigma  \(\frac{\lambda' -\varepsilon u}{p_1 (\kappa-\varepsilon)} -u\) d\mu \nonumber\\
        =& \int_{\Sigma} \frac{\lambda' -\varepsilon u}{p_1 (\kappa)-\varepsilon}d\mu - (n+1)\int_{\Omega} \lambda' dv  \geq 0.\label{integral geq 0}
	\end{align}
	Comparing  \eqref{s5:integral leq 0b} with \eqref{integral geq 0}, we have that equality holds in the second inequality in \eqref{integral geq 0},  which implies that $\Sigma$ is a geodesic sphere. 
	
	Moreover, if either $\partial_1 \chi < 0$ or $\partial_2 \chi>0$, then it follows from \eqref{rigid-chi_1} and \eqref{rigid-chi_2} that $\sum_{i=1}^n \metric{V}{e_i}^2 =0$ everywhere on $\Sigma$. Then \eqref{s2:Phi-i} implies that $\Sigma$ is a geodesic sphere centered at $p$. We complete the proof of Theorem \ref{thm-s1:shifted-eq}.
\end{proof}

\begin{proof}[Proof of Theorem \ref{thm-+1 shifted-eq}]
	Since there exists at least one point on $\Sigma$ at which $h_{ij}>g_{ij}$, it follows from Proposition \ref{prop-suff cond of shifted k-convex} that the assumption $p_k (\kappa+1)>0$ implies $\dot{p}_k^{ij} >0$ on $\Sigma$, and that the inequality \eqref{s2:NM-ineq} is valid. 

    Recall that $\Phi(r)$ is defined as a function satisfying $\Phi'(r) = \lambda(r)$. Hence, we can set $\Phi(r)=\cosh r=\lambda'$ when the ambient space is $\mathbb{H}^{n+1}$. The remainder of the proof of Theorem \ref{thm-+1 shifted-eq} is almost the same as that of Theorem \ref{thm-s1:shifted-eq} by taking $\varepsilon =-1$. The only difference is that instead of Theorem \ref{thm-shifted H-K ineq}, we apply Theorem \ref{thm-+1 HK ineq} in \eqref{integral geq 0}. We complete the proof of Theorem \ref{thm-+1 shifted-eq}.
\end{proof}

\section*{Declarations}

\textbf{Funding}.  The first author was supported by NSFC grant No. 11831005 and NSFC-FWO Grant No. 11961131001.
    The second author was supported by National Key Research and Development Program of China 2021YFA1001800 and 2020YFA0713100, and the Fundamental Research Funds for the Central Universities.
	The research leading to these results is part of a project that has received funding from the European Research Council (ERC) under the European Union's Horizon 2020 research and innovation programme (grant agreement No 101001677). 

\textbf{Competing interests}. The authors have no relevant financial or non-financial interests to disclose.


\begin{bibdiv}
\begin{biblist}
\bib{Alex56}{article}{
	author={Alexandrov, A.D.},
	title={Uniqueness theorems for surfaces in the large I},
	journal={Vestn. Leningr. Univ.},
	volume={11},
	year={1956},
	pages={5--17},
}

\bib{And94}{article}{
   author={Andrews, B.},
   title={Contraction of convex hypersurfaces in Riemannian spaces},
   journal={J. Differential Geom.},
   volume={39},
   date={1994},
   number={2},
   pages={407--431},
}

		
\bib{ACW21}{article}{
	author={Andrews, B.},
	author={Chen, X.},
	author={Wei, Y.},
	title={Volume preserving flow and Alexandrov-Fenchel type inequalities in
		hyperbolic space},
	journal={J. Eur. Math. Soc. (JEMS)},
	volume={23},
	date={2021},
	number={7},
	pages={2467--2509},
}
		
\bib{Bre13}{article}{
	author={Brendle, S.},
	title={Constant mean curvature surfaces in warped product manifolds},
	journal={Publications math\'ematiques de l'IH\'ES},
	volume={117},
	year={2013},
	pages={247--269},
}

\bib{BH17}{article}{
   author={Brendle, S.},
   author={Huisken, G.},
   title={A fully nonlinear flow for two-convex hypersurfaces in Riemannian
   manifolds},
   journal={Invent. Math.},
   volume={210},
   date={2017},
   number={2},
   pages={559--613},
}

		
\bib{BHW16}{article}{
	author={Brendle, S.},
	author={Hung, P.-K.},
	author={Wang, M.-T.},
	title={A Minkowski inequality for hypersurfaces in the anti--de
		Sitter--Schwarzschild manifold},
	journal={Comm. Pure Appl. Math.},
	volume={69},
	date={2016},
	number={1},
	pages={124--144},
}
				
\bib{BW14}{article}{
	author={Brendle, S.},
	author={Wang, M.-T.},
	title={A Gibbons-Penrose inequality for surfaces in Schwarzschild
		spacetime},
	journal={Comm. Math. Phys.},
	volume={330},
	date={2014},
	number={1},
	pages={33--43},
}
		
\bib{CLZ19}{article}{
	author={Chen, D.},
	author={Li, H.},
	author={Zhou, T.},
	title={A Penrose type inequality for graphs over
		Reissner-Nordstr\"{o}m--anti-deSitter manifold},
	journal={J. Math. Phys.},
	volume={60},
	date={2019},
	number={4},
	pages={043503, 12pp},
}

\bib{deLima16}{article}{
   author={de Lima, L. L.},
   author={Girao, F.},
   title={An Alexandrov-Fenchel-type inequality in hyperbolic space with an
   application to a Penrose inequality},
   journal={Ann. Henri Poincar\'{e}},
   volume={17},
   date={2016},
   number={4},
   pages={979--1002},
}

		

\bib{Guan14}{article}{
   author={Guan, P.},
   title={Curvature measures, isoperimetric type inequalities and fully
   nonlinear PDEs},
   conference={
      title={Fully nonlinear PDEs in real and complex geometry and optics},
   },
   book={
      series={Lecture Notes in Math.},
      volume={2087},
      publisher={Springer, Cham},
   },
   date={2014},
   pages={47--94},
}


\bib{GWWX15}{article}{
   author={Ge, Y.},
   author={Wang, G.},
   author={Wu, J.},
   author={Xia, C.},
   title={A Penrose inequality for graphs over Kottler space},
   journal={Calc. Var. Partial Differential Equations},
   volume={52},
   date={2015},
   number={3-4},
   pages={755--782},
}
		
\bib{HK78}{article}{
	author={Heintze, E.},
	author={Karcher, H.},
	title={A general comparison theorem with applications to volume estimates for submanifolds},
	journal={Ann. Sci. Ecole Norm. Sup.},
	volume={11},
	year={1978},
	pages={451--470},
}
		
\bib{HL22}{article}{
	author={Hu, Y.},
	author={Li, H.},
	title={Geometric inequalities for static convex domains in hyperbolic space},
	journal={Trans. Amer. Math. Soc.},
	volume={376},
	number={8},
	pages={5587--5615},
	year={2022},
}
		
\bib{HLW20}{article}{
	author={Hu, Y.},
	author={Li, H.},
	author={Wei, Y.},
	title={Locally constrained curvature flows and geometric inequalities in
		hyperbolic space},
	journal={Math. Ann.},
	volume={382},
	date={2022},
	number={3-4},
	pages={1425--1474},
}
				
\bib{HWZ23}{article}{
	author={Hu, Y.},
	author={Wei, Y.},
	author={Zhou, T.}
	title={A Heintze-Karcher type inequality in hyperbolic space},
	  journal={J. Geom. Anal.},
    volume={34},
    number={4},
	year={2024},
    pages={article no. 113, 17pp},
}
		
		
\bib{LX22}{article}{
	author={Li, H.},
	author={Xu, B.},
	title={Hyperbolic $p$-sum and horospherical $p$-Brunn-Minkowski theory in hyperbolic space},
	eprint={arXiv:2211.06875},
	year={2022},
}
		
\bib{LX19}{article}{
	author={Li, J.},
	author={Xia, C.},
	title={An integral formula and its applications on sub-static manifolds},
	journal={J. Differential Geom.},
	volume = {113},
	number = {2},
	pages = {493--518},
	year = {2019},
}
		
\bib{LN05}{article}{
	author={Li, Y.},
	author={Nirenberg, L.},
	title={The distance function to the boundary, Finsler geometry, and the singular set of viscosity solution},
	year={2005},
	journal={Comm. Pure Appl. Math.},
	volume={58},
	pages={85--146},
}
		
		


\bib{Lyn22}{article}{
   author={Lynch, S.},
   title={Convexity estimates for hypersurfaces moving by concave curvature
   functions},
   journal={Duke Math. J.},
   volume={171},
   date={2022},
   number={10},
   pages={2047--2088},
}

\bib{MR91}{article}{
   author={Montiel, S.},
   author={Ros, A.},
   title={Compact hypersurfaces: the Alexandrov theorem for higher order
   mean curvatures},
   conference={
      title={Differential geometry},
   },
   book={
      series={Pitman Monogr. Surveys Pure Appl. Math.},
      volume={52},
      publisher={Longman Sci. Tech., Harlow},
   },
   date={1991},
   pages={279--296},
}

\bib{PY23}{article}{
	author={Pan, S.},
	author={Yang, B.},
	title={The weighted geometric inequalities for static convex domains in static rotationally symmetric spaces},
	eprint={arXiv:2311.02364},
	year={2023},
}

\bib{QX15}{article}{
	author={Qiu, G.},
	author={Xia, C.},
	title={A generalization of Reilly's formula and its applications to a new
		Heintze-Karcher type inequality},
	journal={Int. Math. Res. Not. IMRN},
	date={2015},
	number={17},
	pages={7608--7619},
}		
		
\bib{Rei77}{article}{
	author={Reilly, R.C.},
	title={Applications of the Hessian operator in a Riemannian manifold},
	journal={Indiana Univ. Math. J.},
	volume={26},
	year={1977},
	pages={459--472},
}
		
\bib{Ros87}{article}{
	author={Ros, A.},
	title={Compact hypersurfaces with constant higher order mean curvatures},
	journal={Rev. Mat. Iberoamericana},
	volume={3},
	date={1987},
	number={3-4},
	pages={447--453},
}


\bib{SX19}{article}{
   author={Scheuer, J.},
   author={Xia, C.},
   title={Locally constrained inverse curvature flows},
   journal={Trans. Amer. Math. Soc.},
   volume={372},
   date={2019},
   number={10},
   pages={6771--6803},
}

		
\bib{Sch14}{book}{
	author={Schneider, R.},
	title={Convex bodies: the Brunn-Minkowski theory},
	series={Encyclopedia of Mathematics and its Applications},
	volume={151},
	edition={Second expanded edition},
	publisher={Cambridge University Press, Cambridge},
	date={2014},
	pages={xxii+736},
}

		
\bib{WWZ22}{article}{
	author={Wang, X.},
	author={Wei, Y.},
	author={Zhou, T.},
	title={Shifted inverse curvature flows in hyperbolic space},
	journal={Calc. Var. Partial Differential Equations},
	volume={62},
	date={2023},
	number={3},
	pages={Paper No. 93, 44 pp},
}
		
		
\bib{Xia16}{article}{
	author={Xia, C.},
	title={A Minkowski type inequality in space forms},
	journal={Calc. Var. Partial Differential Equations},
	volume={55},
	date={2016},
	number={4},
	pages={Art. 96, 8 pp},
}
		
\end{biblist}
\end{bibdiv}
\end{document}